\documentclass[12pt,reqno]{amsart}
\usepackage{amsmath}
\usepackage[dvips]{graphicx}
\usepackage{amsfonts}
\usepackage{amssymb}
\usepackage{latexsym}
\graphicspath{{Img/}}
\newtheorem{theorem}{Theorem}
\theoremstyle{plain}

\newtheorem{corollary}{Corollary}

\newtheorem{definition}{Definition}
\newtheorem{example}{Example}

\newtheorem{lemma}{Lemma}

\newtheorem{problem}{Problem}
\newtheorem{proposition}{Proposition}

\numberwithin{equation}{section}

\begin{document}

\title[An oscillatory Fermat-Torricelli tree]{An oscillatory Fermat-Torricelli tree in $\mathbb{R}^{2}$}
\author{Anastasios N. Zachos}
\address{Greek Ministry of Education, Greece}
\email{azachos@gmail.com} \keywords{oscillation, weighted
Fermat-Torricelli problem, weighted Fermat-Torricelli point,
weighted Fermat-Torricelli tree, oscillatory Fermat-Torricelli
tree} \subjclass{51M04, 51E10, 70B5, 70M20}

\begin{abstract}
We obtain an important generalization of the mechanical solution
given by S. Gueron and R. Tessler w.r. to the weighted
Fermat-Torricelli problem which derives a new structure of
solutions which may be called oscillatory Fermat-Torrlcelli trees.
The weighted Fermat-Torricelli problem in $\mathbb{R}^{2}$ states
that: Given three points in $\mathbb{R}^{2}$  and a positive real
number (weight) which correspond to each point , find the point
(weighted Fermat-Torricelli point) such that the sum of the
weighted distances to these three points is minimized. By applying
the mechanical device of Pick and Polya the oscillatory tree
solution is a new solution w.r to the weighted Fermat-Torricelli
problem for a given isosceles triangle with corresponding two
equal weights at the vertices of the base segment. it is worth
mentioning that at after time $t$ the oscillatory knot of the
mechanical system passes from the weighted Fermat-Torricelli point
with non zero velocity. Furthermore, we give a numerical example
to verify the structure of an oscillatory Fermat-Torricelli tree
for a given isosceles triangle with equal weights.

\end{abstract}\maketitle

\section{Introduction}

We start by stating the weighted Fermat-Torricelli problem in
$\mathbb{R}^{2}:$

\begin{problem}
Given three points $A_{1}=(x_{1},y_{1}),$ $A_{2}=(x_{2},y_{2}),$
$A_{3}=(x_{3},y_{3}),$ find a point $O$ which minimizes the
objective function

\begin{equation}\label{obj1}
f(x,y)=\sum_{i=1}^{3}w_{i}\sqrt{(x-x_{j})^{2}+(y-y_{j})^{2}}
\end{equation}

where $w_{i}$ is a positive real number (weight) which corresponds
to $A_{i}.$

\end{problem}

The solution of the weighted Fermat-Torricelli problem (Problem~1)
is called the weighted Fermat-Torricelli tree, which consists of
the union of the three edges (branches) $A_{1}O,$ $A_{2}O,$
$A_{3}O$ which meet at the weighted Fermat-Torricelli point $O.$

By replacing $w_{1}=w_{2}=w_{3}$ in (\ref{obj1}), we obtain the
(unweighted) Fermat-Torricelli tree. The (unweighted
)Fermat-Torricelli problem was first stated by Pierre de Fermat
(1643) and first solved by E. Torricelli.

The existence and uniqueness of the weighted Fermat-Torricelli
tree and a complete characterization of the "floating case" and
"absorbed case" has been established by Y. S Kupitz and H. Martini
(see \cite{Kup/Mar:97}, theorem 1.1, reformulation 1.2 page 58,
theorem 8.5 page 76, 77). A particular case of this result for
three non-collinear points in $\mathbb{R}^{2},$ is given by the
following theorem:

\begin{theorem}{\cite{BolMa/So:99},\cite{Kup/Mar:97}}\label{theor1}
Let there be given three non-collinear points $A_{1}, A_{2},
A_{3}\in\mathbb{R}^{2}$ with corresponding positive
weights $w_{1}, w_{2}, w_{3}.$ \\
(a) The weighted Fermat-Torricelli point $O$ exists and is
unique. \\
(b) If for each point $A_{i}\in\{A_{1},A_{2},A_{3}\}$

\begin{equation}\label{floatingcase}
\|{\sum_{j=1, i\ne j}^{3}w_{j}\vec u(A_i,A_j)}\|>w_i,
\end{equation}

 for $i,j=1,2,3$  holds,
 then \\
 ($b_{1}$) the weighted Fermat-Torricelli point $O$ (weighted floating equilibrium point) does not belong to $\{A_{1},A_{2},A_{3}\}$
 and \\
 ($b_{2}$)

\begin{equation}\label{floatingequlcond}
 \sum_{i=1}^{3}w_{i}\vec u(O,A_i)=\vec 0,
\end{equation}
where $\vec u(A_{k} ,A_{l})$ is the unit vector from $A_{k}$ to
$A_{l},$ for $k,l\in\{0,1,2,3\}$
 (Weighted Floating Case).\\
 (c) If there is a point $A_{i}\in\{A_{1},A_{2},A_{3}\}$
 satisfying
 \begin{equation}
 \|{\sum_{j=1,i\ne j}^{3}w_{j}\vec u(A_i,A_j)}\|\le w_i,
\end{equation}
then the weighted Fermat-Torricelli point $O$ (weighted absorbed
point) coincides with the point $A_{i}$ (Weighted Absorbed Case).
\end{theorem}

By replacing $w_{1}=w_{2}=w_{3}$ in Theorem~1, we get:

\begin{corollary}\label{cor1}
If $w_{1}=w_{2}=w_{3}$ and all three angles of the triangle
$\triangle A_{1}A_{2}A_{3}$  are less than $120^{\circ},$ then $O$
is the isogonal point (interior point) of $\triangle
A_{1}A_{2}A_{3}$ whose sight angle to every side of $\triangle
A_{1}A_{2}A_{3}$ is $120^{\circ}.$
\end{corollary}\label{cor2}
\begin{corollary}If $w_{1}=w_{2}=w_{3}$ and one of the angles of
the triangle $\triangle A_{1}A_{2}A_{3}$ is equal or greater than
$120^{\circ},$ then $O$ is the vertex of the obtuse angle of
$\triangle A_{1}A_{2}A_{3}.$
\end{corollary}



For an excellent historical exposition regarding the solution of
the weighted Fermat-Torricelli problem we mention the works of
\cite{BolMa/So:99}, \cite{Kup/Mar:97}\cite{Gue/Tes:02} and
\cite{MHAjja:94} and for further generalizations classical works
are given in \cite{IvanTuzh:094} and \cite{IvanovTuzhilin:01}.

In 2002, S Gueron and R. Tessler invented a mechanical solution in
the sense of Polya and Varignon by applying the following
construction (\cite{Gue/Tes:02}):

Suppose that $\{A_{1},A_{2},A_{3}\}$ lie on a horizontal table ,
and that holes are drilled at the vertices, where smooth pulleys
are attached. The three massless strings referring to $A_{1},$
$A_{2}$ and $A_{3}$ that emanate from the knot are passed through
the pulleys, and three masses $w_{1},$ $w_{2},$ $w_{3},$ are
suspended from the ends of these strings.

Assume that the system is released and reaches its mechanical
equilibrium, and that the knot stop at the interior point $O.$
Then, by applying the minimum energy principle at equilibrium,
they obtain the weighted Fermat-Torricelli tree solution. Thus,
the mechanical equilibrium of the system gives the condition for
vectorial balance:

\begin{equation}\label{vecbal1}
w_{1}\vec{u}(O,A_{1})+w_{2}\vec{u}(O,A_{2})+w_{3}\vec{u}(O,A_{3})=\vec{0}.
\end{equation}

We note that the mechanical system reaches its mechanical
equilibrium, by taking into account friction.

In this paper, we generalize the mechanical solution of S. Gueron
and R. Tessler for the case of an isosceles triangle $\triangle
A_{1}A_{2}A_{3}$ where $w_{1}=1$ and $w_{2}=w_{3}$ by introducing
the oscillatory Fermat-Torricell tree solution of the
corresponding mechanical system by assuming it is frictionless.

We note that after time $t$ the oscillatory knot of the mechanical
system passes from the weighted Fermat-Torricelli point with non
zero velocity, by releasing the mechanical system from the vertex
$A_{1}$ with zero velocity
(Theorem~\ref{motion},Corollary~\ref{w2w31}).

Furthermore, we give a numerical example to verify the structure
of an oscillatory Fermat-Torricelli tree for a given isosceles
triangle with equal weights (Example~\ref{exam1}).


\section{A generalization of the mechanical solution of S. Gueron and R. Tessler}

We shall use the same mechanical system of S. Gueron and R.
Tessler in the spirit of Pick and Polya, in order to solve the
following mechanical problem:

\begin{problem}\label{GTPP}
A board is drilled with three holes corresponding to three given
points $A_{1},$ $A_{2},$ $A_{3}$ which form an isosceles triangle
$A_{1}A_{2}A_{3}$ where $A_{1}A_{2}=A_{1}A_{3}=a$ and three
strings are tied together in a knot with mass $m_{0}$ at one knot
and the loose ends are passed through the three holes attached to
the physical weights $w_{1}=1$ from $A_{1},$ $w_{2}$ from $A_{2}$
and $w_{3}$ from $A_{3},$ where $w_{2}=w_{3}.$ If we release
$m_{0}$ from $A_{1}$ with zero velocity find the motion of the
knot $O.$
\end{problem}

\begin{definition}
We call the motion of the knot with mass $m_{0}$ w.r. to the
mechanical system of Problem~\ref{GTPP} an oscillatory
Fermat-Torricelli tree.
\end{definition}

We shall verify the oscillation of the knot with mass $m_{0}$
numerically in example~1.

We denote by $O$ the corresponding weighted Fermat-Torricelli
point of the isosceles triangle $\triangle A_{1}A_{2}A_{3}$ where
$w_{2}=w_{3}$ and $w_{1}=1.$ The point $O$  belongs to the height
$A_{1}A_{4}$ w.r. to the base $A_{2}A_{3}.$ By applying theorem~1
for $w_{2}=w_{3}$ and $w_{1}=1$ we get:

\begin{lemma}\label{lem1}
If $\angle A_{4}A_{1}A_{3} < \frac{\arccos(\frac{1}{2
w_{2}^{2}}-1)}{2},$

then the weighted Fermat-Torricelli point $O$ of $\triangle
A_{1}A_{2}A_{3},$ belongs to the height $A_{1}A_{4}$ w.r. to the
base $A_{2}A_{3}$ and
\begin{equation}\label{eq1}
\angle A_{4}OA_{3}=\frac{ \arccos(\frac{1}{2 w_{2}^{2}}-1)}{2}.
\end{equation}

\end{lemma}

We assume that $\angle A_{4}A_{1}A_{3} < \frac{\arccos(\frac{1}{2
w_{2}^{2}}-1)}{2},$ such that Lemma~1 holds.

Suppose that we release mass $m_{0}$ from the vertex $A_{1}$ with
zero velocity $\dot{x}(0)=0.$ After time $t,$ $m_{0}$ reaches at
the point $S$ which lies on $A_{1}O,$ because $F_{2}=F_{3}=w_{2}$
Thus, the knot will move along the line defined by $A_{1}O$ via
the force $\vec{F}_{23}-\vec{F}_{1},$ such that
$\vec{F}_{23}=\vec{F}_{2}+\vec{F}_{3}$ and $\Delta
F=F_{23}-F_{1}=2\cos\angle OSA_{3}-1$ (see fig.~\ref{fig:1}).

We set $\angle OSA_{3}:=\phi(t),$ $x[t]:=A_{1}S$ and $\angle
OA_{1}A_{3}:=\phi(0).$

\begin{theorem}\label{motion}
The solution of the mechanical system of Problem~\ref{GTPP} is an
oscillatory Fermat-Torricelli tree which is described by the
motion of the knot with mass $m_{0}$ along the line defined by the
$A_{1}O$ having non zero velocity after time $t_{0}$ at the
weighted Fermat-Torricelli point $O$ of $\triangle
A_{1}A_{2}A_{3}$ which is given by:

\begin{equation}\label{nonzerovelocity}
\dot{x}(t_{0})=\sqrt{\frac{2}{m_{0}}(2aw_{2}-x(t_{0})-\frac{2aw_{2}
\sin\phi(0)}{\sin(\angle A_{4}OA_{3})})}.
\end{equation}

\end{theorem}

\begin{proof}
\begin{figure}
\centering
\includegraphics[scale=1.25]{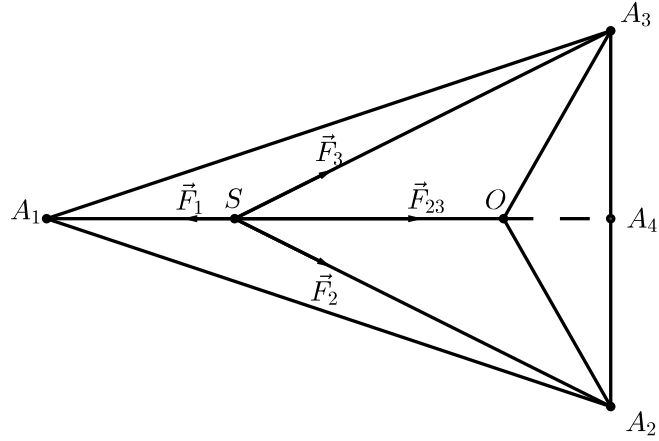}
\caption{Motion of the oscillatory Fermat-Torricelli tree $A_{1}S,
A_{2}S, A_{3}S$ along $A_{1}O$ for the boundary isosceles
$\triangle A_{1}A_{2}A_{3}$} \label{fig:1}
\end{figure}
By applying the sine law in $\triangle SOA_{2},$ we obtain:

\begin{equation}\label{sinelaw1}
\frac{OS}{\sin(\frac{\angle A_{4}OA_{3}}{2})-\phi(t)
}=\frac{OA_{3}}{\sin\phi(t)}
\end{equation}

where

\begin{equation}\label{calcOS}
OS=a(\cos\phi(0)-\frac{\sin\phi(0)}{\cot(\frac{\angle
A_{4}OA_{3}}{2})})-x
\end{equation}

By replacing (\ref{calcOS}) in (\ref{sinelaw1}), we get:

\begin{equation}\label{calcxt}
x(t)=a \cos\phi(0)-a\sin\phi(0)\cot\phi(t).
\end{equation}

At time $t$ the force along the path $A_{1}O$ is given by:
\begin{equation}\label{forcepath}
m_{0}\ddot{x}=2w_{2}\cos\phi(t)-1.
\end{equation}

By differentiating (\ref{calcxt}), we have:

\begin{equation}\label{difx}
dx=\frac{a\sin\phi(0)}{\sin^{2}\phi(t)}d\phi
\end{equation}

It is well known that:

\[\ddot{x}=\dot{x}\frac{d\dot{x}}{dx}.\]

Thus, by integrating both parts of (\ref{forcepath}) from $A_{1}$
to $S,$ w.r. to $x$ and taking into account (\ref{calcxt}), we
obtain:

\begin{equation}\label{integrate1}
m_{0}\frac{(\dot{x}(t))^{2}-\dot{x}(0))^{2}}{2}=\int_{\phi(0)}^{\phi}2aw_{2}\sin\phi(0)\frac{\cos\phi}{\sin^{2}\phi(0)}d\phi
-\int_{0}^{x}dx
\end{equation}

or

\begin{equation}\label{integrate2}
m_{0}\frac{(\dot{x}(t))^{2}}{2}=2aw_{2}-2aw_{2}\frac{\sin\phi(0)}{\sin\phi}
-\int_{0}^{x}dx
\end{equation}

By setting $t=t_{0}$ in (\ref{integrate2}), we obtain
(\ref{nonzerovelocity})

\end{proof}

\begin{corollary}\label{w2w31} For $w_{2}=w_{3}=1,$ the velocity
of the knot with mass $m_{0}$ which passes from the unweighted
Fermat-Torricelli point is given by:
\begin{equation}\label{velocity2}
\dot{x}(t_{0})=\sqrt{\frac{2}{m_{0}}(2-x(t_{0})-\frac{4a
\sin\phi(0)}{\sqrt{3}})}.
\end{equation}
\end{corollary}

\begin{proof}
By replacing $w_{2}=w_{3}=1,$ and $\phi=60^{\circ},$ we derive
(\ref{velocity2}).
\end{proof}

\begin{proposition}\label{prop1} The movement of the mechanical
system is determined by the following differential equation:
\begin{equation}\label{differential1}
m_{0} (\frac{a
\sin\phi(0)}{\sin^{2}\phi(t)}\ddot{\phi}-2a\frac{\sin\phi(0)}{\sin^{3}\phi}\cos\phi\dot{\phi})=2w_{2}\cos\phi-1
\end{equation}
with initial conditions $\phi(0)=\phi_{0}$ and $\dot{\phi}=0.$
\end{proposition}

\begin{proof}
Differentiating twice (\ref{calcxt}) and by replacing in
(\ref{forcepath}) we derive (\ref{differential1}).
\end{proof}

\begin{proposition}\label{prop2}
The work of the force of the mechanical system along the path
$A_{1}O$ starting from $A_{1}$ with $\dot{x}(0)=0$ is given by:
\begin{equation}\label{work}
W=2w_{2}(a-A_{2}O)-A_{1}O\ne 0.
\end{equation}
\end{proposition}
\begin{proof}

We start with the work of the force $F_{23}-F_{1}$ along $A_{1}O:$
\begin{equation}\label{workf1}
W=\int_{0}^{A_{1}O}(F_{23}-F_{1})dx.
\end{equation}

By replacing (\ref{difx}) in (\ref{workf1}), we get:
\begin{equation}\label{workf2}
W=\int_{\phi(0)}^{\angle A_{4}OA_{3}}
2w_{2}\cos\phi\frac{a\sin\phi(0)}{\sin^{2}\phi}d\phi -A_{1}O,
\end{equation}

or

\[W= 2w_{2}(a-a\frac{\sin\phi(0)}{\sin\angle A_{4}OA_{3}})-1\]

which yields (\ref{work}).

\end{proof}

For $w_{2}=w_{3}=1,$ we get:
\begin{corollary}\label{cor3}
The work of the force of the mechanical system along the path
$A_{1}O$ starting from $A_{1}$ with $\dot{x}(0)=0$ is given by:
\begin{equation}\label{work2}
W=2(a-A_{2}O)-A_{1}O\ne 0,
\end{equation}
where $O$ is the unweighted Fermat-Torricelli point.
\end{corollary}

\begin{example}\label{exam1}

Given an isosceles triangle $\triangle A_{!}A_{2}A_{3}$ where
$a=5,$ $\phi(0)=40^{\circ},$ $w_{1}=w_{2}=w_{3}=1,$ we derive that
$\angle A_{4}OA_{3}=60^{\circ}.$

\begin{figure}
\centering
\includegraphics[scale=0.70]{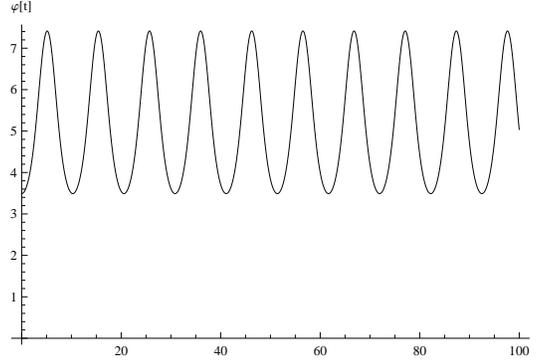}
\caption{Graph of $\phi(t)$ for $a=5, \phi(0)=40^{\circ}$}
\label{fig:2}
\end{figure}
\begin{figure}
\centering
\includegraphics[scale=0.65]{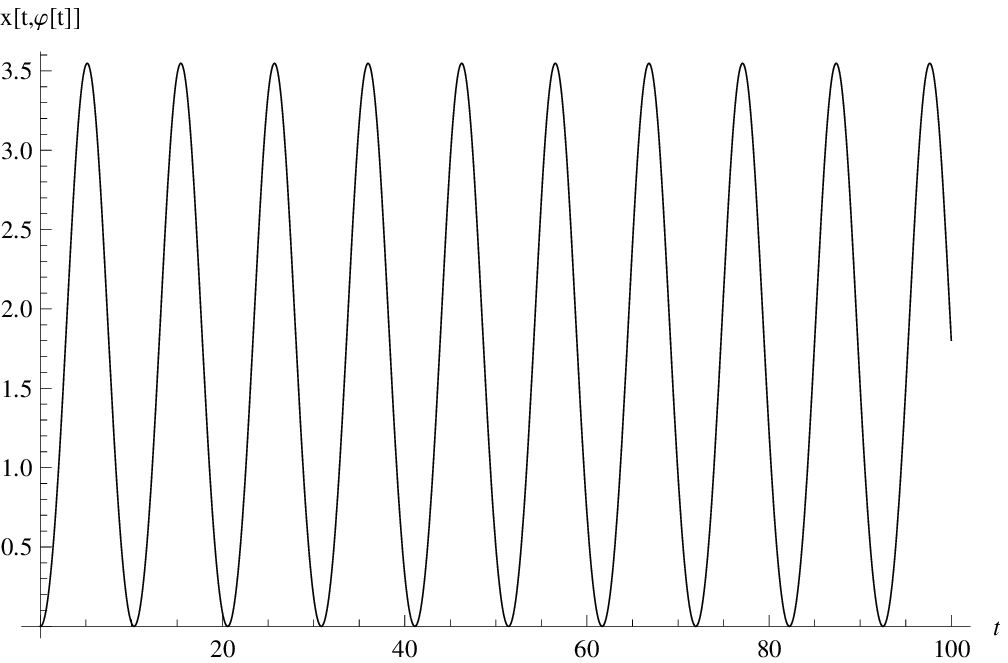}
\caption{Graph of x(t), for $a=5, \phi(0)=40^{\circ}$}
\label{fig:3}
\end{figure}
\begin{figure}
\centering
\includegraphics[scale=0.65]{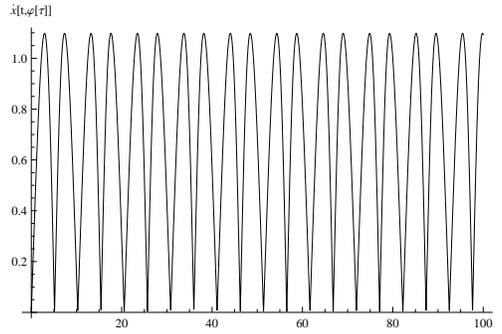}
\caption{Graph of $\dot{x}$(t) for $a=5,
\phi(0)=40^{\circ},m_{0}=1.$} \label{fig:4}
\end{figure}

Suppose that we release mass $m_{0}$ from the vertex $A_{1}$ with
zero velocity $\dot{x}(0)=0.$ After time $t,$ $m_{0}$ reaches at
the point $S$ which lies on $A_{1}O.$ By replacing $a,\phi(0)$ in
(\ref{differential1}), we obtain a numerical solution using
Mathematica of $\phi(t)$ and $x(t)$ (see
fig.~\ref{fig:2},~\ref{fig:3} ).

By replacing $a,\phi(0),m_{0}=1$ in (\ref{velocity2}), we derive a
numerical solution using Mathematica of $\dot{x}(t)$ (see
fig.~\ref{fig:4}).

We note that an approximation of the periodical function $x(t)$
may be of the form $a \sin(b (t+c))+d:$

\[x[t]\approx 1.77363+ 1.77363 \sin(0.61133 (-2.56947 + t))\]

By differentiating $x(t)$ w.r. to $t,$ we get a good approximation
of $\dot{x}(t):$

\[\dot{x}(t)\approx |1.08427 \cos(0.61133 (-2.56947 + t))|\]

\begin{figure}
\centering
\includegraphics[scale=0.65]{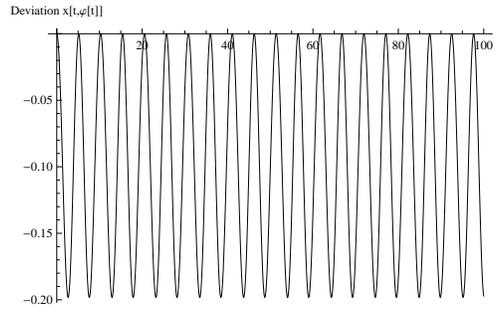}
\caption{Graph of $x(t)-(1.77363+ 1.77363 \sin(0.61133 (-2.56947 +
t)))$, for $a=5, \phi(0)=40^{\circ}$} \label{fig:5}
\end{figure}
\begin{figure}
\centering
\includegraphics[scale=0.65]{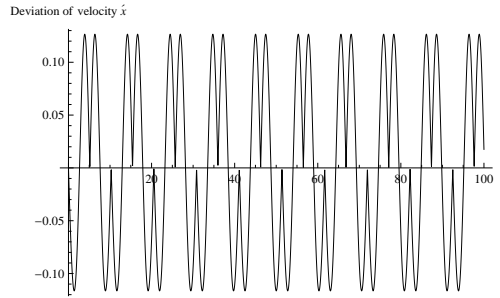}
\caption{Graph of $\dot{x}$(t)-$|1.08427 \cos(0.61133 (-2.56947 +
t))|$ for $a=5, \phi(0)=40^{\circ},m_{0}=1.$} \label{fig:6}
\end{figure}

The deviation of $x(t)$ is given by fig.~\ref{fig:5} and the
deviation of the corresponding velocity $\dot{x}(t)$ is given by
fig.~\ref{fig:6}.

\end{example}

\end{document}